\documentclass[10pt]{amsart}
\usepackage{amssymb, amstext, amscd, amsmath,color}

\theoremstyle{plain}
\newtheorem{thm}{Theorem}[section]
\newtheorem{cor}[thm]{Corollary}
\newtheorem{prop}[thm]{Proposition}
\newtheorem{lem}[thm]{Lemma}

\theoremstyle{definition}
\newtheorem{rem}[thm]{Remark}

\newtheorem{eg}[thm]{Example}

\newcommand{\bA}{{\mathbb{A}}}
\newcommand{\bB}{{\mathbb{B}}}
\newcommand{\bC}{{\mathbb{C}}}
\newcommand{\bD}{{\mathbb{D}}}

  \newcommand{\A}{{\mathcal{A}}}
  \newcommand{\B}{{\mathcal{B}}}

  \newcommand{\I}{{\mathcal{I}}}

  \newcommand{\M}{{\mathcal{M}}}

  \newcommand{\U}{{\mathcal{U}}}

\renewcommand{\phi}{\varphi}
\newcommand{\upchi}{{\raise.35ex\hbox{\ensuremath{\chi}}}}

\newcommand{\ol}{\overline}

\usepackage{microtype}

\title{Pick interpolation in several variables}
\author[R. Hamilton]{Ryan Hamilton}
\address{Pure Math.\ Dept.\\U. Waterloo\\Waterloo, ON\;
N2L--3G1\\CANADA}
\email{rhamilto@uwaterloo.ca}
\date{\today}
\thanks{R.~Hamilton is partially supported by an NSERC graduate scholarship}

\begin{document}

\begin{abstract}
We investigate the Pick problem for the polydisk and unit ball using dual algebra techniques. 
Some factorization results for Bergman spaces are used to describe a Pick theorem for any bounded region in $\mathbb{C}^d$.
\end{abstract}

\maketitle

\section{Introduction}
Suppose $z_1, \dots, z_n$ are distinct points in the complex open disk $\bD$ 
and  $w_1, \dots, w_n$ are complex numbers. 
The classic Pick interpolation theorem for $\bD$ states the following:
{\em there is a holomorphic function $f$ on $\bD$ satisfying $f(z_i) = w_i$ for $i=1, \dots ,n$ and 
$\sup \{ | f(z) |: z \in \bD \} \leq 1$ if and only if the matrix $\left [( 1-w_i \ol{w_j})(z_i\ol{z_j})^{-1} \right ] $ is positive 
semidefinite}.

It is well known that, in general, the analogue of Pick's theorem does not hold for domains other than the disk. The single matrix present in Pick's theorem is
 typically replaced with an infinite family of matrices. 
Abrahamse's  interpolation theorem  for multiply connected regions \cite{Abram79} was the first
appearance of this phenomenon, where the family of matrices 
are naturally parametrized by a polytorus. 
  Cole, Lewis and Wermer \cite{CLW} approached the Pick problem in substantial generality by considering the problem for any uniform algebra. 
  They showed that a solution exists when the Pick matrices associated to a large class of measures are all positive semidefinite. In \cite{AMc99}, Agler and McCarthy carried out a deep analysis of the Pick problem for the bidisk.
Though an infinite family of Pick matrices once again appeared, their simultaneous positivity was reforumulated
as a factorization problem for $H^\infty(\bD^2)$.  These results also hold for the polydisk, but the associated Pick theorem
is not given in terms of the $H^\infty(\bD^d)$ norm. See \cite{AMc02} and the references therein for
a detailed treatment of all of these results.

 In this paper, we present
two Pick-type theorems in a multivariable setting.
For $d \geq 2$, the open unit ball and polydisk in $\bC^d$ will be denoted $\bB_d$ and $\bD^d$, respectively. 
For a bounded domain $\Omega \subset \bC^d$ and Lebesgue measure $\mu$ on $\Omega$, the Bergman space $L_a^2(\Omega)$ is defined as those functions which are analytic on $\Omega$ and contained in $L^2(\Omega, \mu)$.
If $\Omega$ is either $\bB_d$ by $\bD^d$, the Hardy space $H^2(\Omega)$ is defined as the closure of the multivariable analytic polynomials
in $L^2(\partial \Omega, \theta)$, where $\theta$ is Lebesgue measure on $\partial \Omega$.  The algebra of bounded analytic functions on $\Omega$ will be denoted $H^\infty(\Omega)$. 
In Theorem~\ref{T:thm1} below, a Pick theorem for the polydisk and unit ball is obtained. Theorem~\ref{T:thm1} improves upon existing interpolation
theorems for these domains by explicitly describing the associated family of Pick matrices in terms of a certain class of absolutely continuous measures. 
Moreover, no distinction is made between the $d=2$ and $d > 2$ case.
In Theorem~\ref{T:thm2} a Pick result for any bounded domain
in $\bC^d$ is established using Bergman spaces. As is the case with Theorem~\ref{T:thm1}, the associated Pick matrices arise from absolutely continuous measures
on $\Omega$. Both theorems are also valid for any weak$^*$-closed subalgebra of $H^\infty(\Omega)$, a feature which
is typically absent in Pick interpolation results.
In both theorems, the symbol $k^\nu$ refers to a reproducing kernel function described in Section~\ref{S:Alg}.

\begin{thm}\label{T:thm1}
Suppose $\Omega$ is either $\bD^d$ or $\bB_d$ and that $z_1, \dots, z_n \in \Omega$ and $w_1, \dots, w_n \in \bC$. Let $\A$ be any weak$^*$-closed
subalgebra of $H^\infty(\Omega)$.
There is a  function $\phi \in \A$ with $\sup_{z \in \Omega} |\phi(z)| \leq 1$ and $\phi(z_i) = w_i$ for $i=1, \dots,  n$ if and only if the matrix
\[ 
\left [
(1 - w_i\ol{w_j})k^\nu(z_i,z_j)
\right ]_{i,j=1}^n \geq 0
\]
is positive semidefinite for every measure of the form $\nu = |f|^2\theta$, where $\theta$ is Lebesgue measure on $\partial \Omega$ and $f \in  H^2(\Omega)$.
\end{thm}

\begin{thm} \label{T:thm2}
Suppose $\Omega$ is a bounded domain in $\bC^d$ and that $z_1, \dots, z_n \in \Omega$ and $w_1, \dots, w_n \in \bC$.
Let $\A$ be any weak$^*$-closed
subalgebra of $H^\infty(\Omega)$.
There is a  function $\phi \in \A$  with $\sup_{z \in \Omega} |\phi(z)| \leq 1$ and $\phi(z_i) = w_i$ for $i=1, \dots,  n$ if and only if the matrix
\[ 
\left [
(1 - w_i\ol{w_j})k^\nu(z_i,z_j)
\right ]_{i,j=1}^n \geq 0
\]
is positive semidefinite for every measure of the form $\nu = |f|^2\mu$, where $\mu$ is Lebesgue measure on $\Omega$ and $f \in  L_a^2(\Omega)$.
\end{thm}

Both Theorem~\ref{T:thm1} and Theorem~\ref{T:thm2} provide significant simplifications of the Cole-Lewis-Wermer 
approach for the algebra $H^\infty(\Omega)$ and its subalgebras. Even though an infinite family of kernel functions
is still required, they are given a concrete description. 
These theorems will be established using some recent results of Davidson and the author
\cite{DH}, where a general framework for Pick-type theorems was developed using the theory of dual
operator algebras and their preduals. Dual algebra techniques in Pick interpolation
 can also be seen in the paper of McCullough \cite{McC}.

In Section~\ref{S:Alg}, we will briefly
discuss reproducing kernel Hilbert spaces and their multipliers. A natural notion of equivalence between reproducing kernel
Hilbert spaces is introduced and applied to cyclic subspaces of $H^2(\Omega)$ and $L^2_a(\Omega)$ in Theorem~\ref{T:kernel}.
In order to prove the desired results, some dual algebra results of Bercovici-Westood \cite{BW} and Prunaru \cite{Prun} are invoked which, when combined
with the other results in Section~\ref{S:Alg}, establish the desired results.

\section{Reproducing kernel Hilbert spaces and their multipliers} \label{S:Alg}

We say that a Hilbert space $H$ of $\bC$-valued functions on $X$ is a \emph{reproducing kernel Hilbert space} if point evaluations are continuous.
For every $x \in X$, the \emph{reproducing kernel at $x$} is the function $k_x \in H$ which satisfies $f(x) = \langle f, k_x \rangle$ for $f \in H$. The associated positive definite
kernel on $X \times X$ is given by $k(x,y) := \langle k_y,k_x\rangle$. A \emph{multiplier} $\phi$ of $H$ is a function on $X$
such that $\phi f \in H$ for every $f \in H$.  Each multiplier $\phi$ induces a bounded multiplication operator $M_\phi$ on $H$. 
The \emph{multiplier algebra of $H$}, denoted $M(H)$, is the
operator algebra consisting of all such $M_\phi$. The multiplier algebra is unital, maximal abelian and closed in the weak operator topology. The adjoints of multiplication 
operators are characterized by the fundamental identity $M_\phi^*k_x = \ol{\phi(x)}k_x$.

\begin{eg}
The kernel functions for the Hardy spaces $H^2(\bD^d)$ and $H^2(\bB_d)$ are given by
\[k^{H^2(\bD^d)}(z,w) = \prod_{k=1}^d\frac{1}{1 - z_i\ol{w_i}}\: \text{  and  } \: k^{H^2(\bB_d)}(z,w) = \frac{1}{\left (1 -\langle z, w \rangle_{\bC^d}\right )^d} , \]
respectively. When $\Omega$ is either $\bD^d$ or $\bB_d$, the multiplier algebra $M(H^2(\Omega))$ is equal to $H^\infty(\Omega)$, and this identification is isometric: $\| M_\phi \| = \| \phi \|_\infty$ for $\phi \in H^\infty(\Omega)$.  If $\Omega$ is any bounded domain in $\bC^d$, the kernel function for the Bergman spaces $L^2_a(\Omega)$ is generally difficult to compute.  Some familiar examples are given by
\[k^{L_a^2(\bD^d)}(z,w) = \prod_{k=1}^d\frac{1}{(1 - z_i\ol{w_i})^2}\: \text{  and  } \: k^{L^2_a(\bB_d)}(z,w) = \frac{1}{\left (1 -\langle z, w \rangle_{\bC^d}\right)^{d+1}} . \]
It is easy to verify that $M(L^2_a(\Omega) = H^\infty(\Omega)$ and that $\|M_\phi\| = \|\phi\|_\infty$ for any $\phi \in H^\infty(\Omega)$.  See the book of Krantz \cite{Krantz} for detailed treatment of these spaces. 
\end{eg}
This note will concern itself only with reproducing kernel Hilbert spaces of analytic functions.  We further assume that $H$ is endowed an $L^2$ norm, i.e. there is some
set $\Delta$ along with a $\sigma$-algebra of subsets and measure $\mu$ such that $H$ is a closed subspace of $L^2(\Delta, \mu)$. 
The set $\Delta$ may play different roles depending on the context.  For the Hardy space of the polydisk or unit ball, $\Delta$ is taken to be either $\partial \bD^2$ or $\partial \bB_d$, respectively and $\mu$ is the corresponding Lebesgue measure on these sets. For the Bergman space $L^2_a(\Omega)$, we simply take $\Delta = \Omega$ and $\mu$ to be Lebesgue measure on $\Omega$.
We also assume that $H$ contains the constant function $1$, so that every multiplier of $H$ is contained in $H$.

Given an algebra of multipliers $\A$ on $H$ and a measure $\nu$ on $\Delta$, let $\A^2(\nu)$ denote the closure of $\A$ in $L^2(\Delta, \nu)$. The measure $\nu$ is said to be \emph{dominating} for $X$ (with respect to $\A$) if $\A^2(\nu)$ is a reproducing kernel Hilbert space on $X$. We will write $k_x^{A^2(\nu)}$ for the reproducing kernel on this space, or more simply as $k_x^\nu$ when the context is clear.
The associated positive definite kernel function on $X \times X$ will be denoted $k^\nu(x,y) := \langle k_y^\nu, k_x^\nu \rangle_{A^2(\nu)}$.
For any such $\nu$, $\A$ is obviously an algebra of multipliers on $\A^2(\nu)$. 
\begin{rem}
Our notion of a dominating measure differs slightly from Cole-Lewis-Wermer \cite{CLW}.  In their setting, $\A$ is
a uniform algebra and $\Delta$ is the maximal idea space of $\A$. A measure $\mu$ on $\Delta$ is said to be \emph{dominating} for
a subset $\Lambda$ of $\Delta$ if there is a constant $C$ such that $|\phi(\lambda)| \leq C\|\phi\|_{L^2(\mu)}$ for every $\lambda \in \Lambda$.
Their theorem solves the so-called \emph{weak} Pick problem for $\A$:
Given $\epsilon > 0$, $w_1, \dots, w_n \in \bC$ and $\lambda_1, \dots, \lambda_n \in \Delta$, there is a function $\phi \in \A$
with $\phi(\lambda_i) = w_i$ for $i=1,\dots,n$ and $\| \phi \| < 1 + \epsilon$ if and only if
\[[(1-w_i\ol{w_j})k^\mu(\lambda_i,\lambda_j)] \geq 0 \]
for every  measure $\mu$ which is dominating for $\{\lambda_1, \dots, \lambda_n \}$.
\end{rem}

A unital, weak-$*$ closed subalgebra of $B(H)$ will be called a \emph{dual algebra}.  The predual $\A_*$ of
a dual algebra may be identified canonically with a quotient of the trace class operators on $H$.
A dual algebra $\A$ is said to have property $\bA_1(1)$ if every $\pi \in \A_*$ with $\| \pi \| < 1$ may be written as 
\[\pi(A) = \langle Ax,y \rangle_H \]
for some $x,y \in H$ which satisfy $\|x\|\|y\| < 1$. See the manuscript of Bercovici, Foia\c s and Pearcy \cite{ABFP} for a detailed treatment
of dual algebras and the structure of their preduals.

If $H$ is a reproducing kernel Hilbert space
and $\A \subset M(H)$ is a dual algebra, we call $\A$ a \emph{dual algebra of multipliers}.
If $L$ is an invariant subspace of $\A$, then $L$ is also a reproducing kernel Hilbert space
with the kernel function at $x$ given by $P_Lk_x$.  Clearly $\A$ is also a dual algebra of multipliers on $L$, 
and we denote the multiplication operator associated to $\phi$ as $M_f^L$.  It follows that
$(M^L_\phi)^*P_Lk_x = \ol{\phi(x)}P_Lk_x$
for $x \in X$.
Note that if $f(x) = 0$ for every $f \in L$, then $P_Lk_x = 0$.  Since we are principally concerned
with evaluation of multipliers, it is useful to extend the kernel function for $L$ to points
which are annihilated by every function in $L$. If $f \in H$, we denote
the closed cyclic subspace generated by $\A$ and $f$ as $\A[f]$.  The following lemma appears as Lemma 2.1 in  \cite{DH}.

\begin{lem}
Suppose $\A$ is a dual algebra of multipliers on a reproducing kernel Hilbert space $H$ and let $\I_x$ denote the ideal
of functions in $\A$ which vanish on $x$.  If $f \in H$ and $\A[f] \neq I_x[f]$, then the one-dimensional subspace
$\A[f] \ominus I_x[f]$ contains a non-zero vector $k_x^f$ such that $P_{\A[f]}k_x = k_x^f$ if $f(x) \neq 0$ and 
$(M_\phi^f)^*k_x^f = \ol{\phi(x)}k_x^f$.
\end{lem} 

 If $L = \A[f]$ for $f \in H$, we will use the notation
$k^f_x$ and $M^f_\phi$ for $k^L_x$ and $M_\phi^L$, respectively. 
The following result of Davidson and the author (\cite{DH}, Theorem 3.4) gives a Pick theorem 
for dual algebras of multipliers which have property $\bA_1(1)$.

\begin{thm} \label{T:DH}
Suppose that  $H$  is both a reproducing kernel Hilbert space 
over a set $X$ and  that $\A$ is a dual algebra of multipliers on 
$H$ which has property $\bA_1(1)$. Then the following statement holds:
given $x_1, \dots, x_n \in X$ and $w_1, \dots, w_n \in \bC$, there is a multiplier $\phi \in \A$
such that $\phi(x_i) = w_i$ and $\| M_\phi \| \leq 1$ if and only if the matrix
\[ 
\left [
(1 - w_i\ol{w_j})\langle k_{x_j}^f,k_{x_i}^f \rangle
\right ] 
\]
are positive semidefinite for every $f \in H$.
\end{thm}

When $H$ is contained in an ambient $L^2$ space, we seek a nicer description of the cyclic subspaces $\A[f]$. 
Suppose $H$ and $K$ are reproducing kernel Hilbert spaces on $X$ with kernels $k$ and $j$, respectively, and that
$U : H \rightarrow K$ is a unitary map. We say that $U$ is a \emph{reproducing kernel Hilbert space isomorphism} if for every $x \in X$
there is a non-zero scalar $c_x$ such that $Uk_x = c_xj_x$.  
We require the following easy proposition, the details of which may be found in (\cite{AMc02}, Section 2.6).
\begin{prop}
Suppose $H$ and $K$ are reproducing kernel Hilbert spaces on a set $X$ and that
$U : H \rightarrow K$ is a reproducing kernel Hilbert space isomorphism. 
As sets, the multiplier algebras of $H$ and $K$ are equal. Moreover, the map $U$ induces a unitary equivalence between $M(H)$ and $M(K)$.
\end{prop}

The following result shows
that cyclic subspaces may naturally be identified with $H$ under a different norm. Recall that if $H \subset L^2(\Delta, \mu)$, $\A \subset M(H)$ and $\nu$ is some measure on $\Delta$
, then $\A^2(\nu)$ is the closure of $\A$ in $L^2(\Delta, \nu)$.

\begin{thm} \label{T:kernel}
Suppose $H$ is a reproducing kernel Hilbert space of analytic functions on some bounded domain $\Omega$. Suppose further that
there is a measure space $(\Delta, \mu)$ such that $H$ is a closed subspace
of $L^2(\Delta,\mu)$.
 If $\A$ is any dual algebra of multipliers on $H$, then $\A^2(|f|^2\mu)$ is a reproducing kernel Hilbert space on the set
 $\Omega_f := \{z \in \Omega: \langle f, k_z^f \rangle_H \neq 0 \}$.  The reproducing kernel for $\A^2(|f|^2\mu)$ is given by
 \[j^f(z,w) = \frac{k^f(z,w)}{\langle f , k_z^f \rangle_H \ol{\langle f, k_w^L \rangle}_H}.\]
Moreover, there is a reproducing kernel Hilbert space isomorphism
\[U : \A[f] \rightarrow \A^2(|f|^2\mu), \]
In particular, any measure of the form $|f|^2\mu$ is dominating for $\Omega_f$ with respect to $\A$.
\end{thm}
\begin{proof}
Define a linear map $V: \A \rightarrow \A[f]$ by $ V\phi = \phi f$.  It is clear by the definition of the norms involved that $V$ extends to a unitary
\[ V: \A^2(|f|^2\mu) \rightarrow \A[f]. \]
We claim that $U := V^*$  is the required isomorphism. For notational convenience,
let $\A^2 := \A^2(|f|^2\mu)$.  If $\phi \in \A$ and $z \in \Omega_f$, we have
\[\langle \phi, V^*k_z^f \rangle_{A^2} = \langle \phi f , k_z^f \rangle_H = \left \langle f, (M_\phi^f)^*k_z^f \right \rangle_H = \phi(z)\langle f, k_z^f \rangle_H.  \]
By assumption, $\langle f, k_z^f \rangle_H \neq 0$, and so the vector $\ol{\langle f, k_z^f \rangle}^{-1}_HV^*k_z^f$  is the reproducing kernel at the point $z$
for any function in $\A$.  
In order to show this for any function $\phi \in \A^2$, find $\phi_n \in \A$ with $\| \phi_n - \phi \|_{\A^2}$ tending to $0$.  Since $\phi \in L^2(\Delta,|f|^2\mu)$,
it follows that $\phi f \in L^2(\Delta, \mu)$ . This in turn implies that the sequence $\{\phi_nf \}$ is Cauchy in $L^2(\Delta, \mu)$ and that its limit must be $\phi f$.  The subspace $\A[f]$ is closed, and so $\phi f \in \A[f]$. Taking an inner product against $k_z^f$ implies
that $\phi_n(z) \rightarrow \phi(z)$ for any $z \in \Omega_f$.  Consequently
\begin{eqnarray*}
\left \langle \phi , \frac{V^*k_z^f}{\ol{\langle f, k_z^f \rangle_H}} \right  \rangle_{\A^2} &=& \left \langle \phi f  , \frac{k_z^f}{\ol{\langle f, k_z^f \rangle_H}}  \right \rangle_H  \\
 &=& \lim_{n \rightarrow \infty} \left \langle \phi_n f  , \frac{k_z^f}{\ol{\langle f, k_z^f \rangle_H}}  \right \rangle_H
 = \lim_{n \rightarrow \infty} \phi_n(z) = \phi(z).  
 \end{eqnarray*}
 Now set $j^f_z := (\ol{\langle f, k_z^f \rangle}_H)^{-1}V^*k_z^f$.  The above reasoning shows that $j_z^f$ is the reproducing kernel for $\A^2$, and 
 \[j^f(z,w) := \langle j_w^f, j_z^f \rangle_{\A^2} =  \frac{\langle V^* k_w^f, V^*k_v^f \rangle_{\A^2}}{\langle f , k_z^f \rangle_H \ol{\langle f, k_w^f \rangle}_H} =  \frac{k^f(z,w)}{\langle f , k_z^f \rangle_L \ol{\langle f, k_w^f \rangle}_H}, \]
 which proves the theorem.
\end{proof}

\begin{cor}\label{C:kernel}
Suppose $\A$ and $H$ satisfy the hypotheses of Theorem~\ref{T:kernel}. Then the matrix 
\[\left [
(1 - w_i\ol{w_j})k^f(z_i,z_j)
\right ]_{i,j=1}^n  \]
is positive semidefinite if and only if 
\[ \left [
(1 - w_i\ol{w_j})j^f(z_i,z_j)
\right ]_{i,j=1}^n  \] 
is positive semidefinite.
\end{cor}
\begin{proof}
By Theorem~\ref{T:kernel}, the matrix $[(1 - w_i\ol{w_j})j^f(z_i,z_j)]$ is the Schur product of $[(1 - w_i\ol{w_j})k^f(z_i,z_j)]$ and $\left [ \left (\langle f , k_{z_j}^f \rangle_L \ol{\langle f, k_{z_i}^f \rangle}_L \right )^{-1} \right ]$, the latter of which is manifestly positive semidefinite.
\end{proof}

We can now summarize the results of this section so far.

\begin{thm}\label{T:main}
Let $(\Delta, \B, \mu)$ be a $\sigma$-finite measure space such that  $H$  is both a reproducing kernel Hilbert space 
over a set $X$ and a closed subspace of $L^2(\Delta,\mu)$. Suppose that $\A$ is a dual algebra of multipliers on 
$H$ which has property $\bA_1(1)$. Then the following statement holds:
given $x_1, \dots, x_n \in X$ and $w_1, \dots, w_n \in \bC$, there is a multiplier $\phi \in \A$
such that $\phi(x_i) = w_i$ and $\| M_\phi \| \leq 1$ if and only if the matrix
\[ [(1-w_i\ol{w_j})k^\nu(x_i,x_j)]_{i,j=1}^n\]
is positive semidefinite for every measure of the form $\nu = |f|^2\mu$ where $ f \in H$.
\end{thm}
\begin{proof}
If such a $\phi$ exists, then $M^f_\phi$ is a contraction for every $f \in H$. It follows that the operator
\[I - M_\phi^f(M^f_\phi)^* \]
is positive semidefinite. Taking an inner product against finite spans of the functions $k_z^f$ implies that 
the matrices \[\left [
(1 - w_i\ol{w_j})k^f(z_i,z_j)
\right ]_{i,j=1}^n \]
are positive semidefinite. Now apply Corollary~\ref{C:kernel} and note that $k^{|f|^2\mu} = j^f$.
The nontrivial direction is obtained by combining Theorem~\ref{T:DH} and Corollary~\ref{C:kernel}.
\end{proof}

We will use the following important factorization result 
of Bercovici-Westwood (\cite{BW}, Theorem 1).
Recall that
$\theta$ is Lebesgue measure on $\partial \Omega$.

\begin{thm}[Bercovici-Westwood]
Suppose $\Omega$ is either $\bD^d$ or $\bB_d$. Then for any function 
$h \in L^1(\partial \Omega, \theta)$ and $\varepsilon > 0$, there are functions $f$ and $g$ in $H^2(\Omega)$ such that \
$\| f \|_2 \|g\|_2 \leq \|h\|_1$ and $\|f - g \ol{h} \|_2 < \varepsilon$.
In particular, $\M(H^2(\Omega))$ has property $\bA_1(1)$.
\end{thm}
Property $\bA_1(1)$ is hereditary for weak-$*$ closed subspaces (\cite{ABFP}, Proposition 2.04), and so we may combine  the above result with Theorem~\ref{T:main}, which proves
Theorem~\ref{T:thm1}.
In order to prove Theorem~\ref{T:thm2}, we employ
a versatile factorization theorem of Punraru (\cite{Prun}, Theorem 4.1). This result applies to any instance where $H$ is a reproducing kernel Hilbert space
on a measure space $(X, \B, \mu)$ and $H$ is a closed subspace of $L^2(X,\mu)$. In particular, it applies to any Bergman space, but not to Hardy space. Any reproducing kernel Hilbert space which satisfies these hypothesis always has the property that $\|M_\phi\| = \sup_{x \in X}|\phi(x)|$ for any multiplier $\phi$.

\begin{thm} [Prunaru] \label{T:prunaru}
Let $(X, \B, \mu)$ be a $\sigma$-finite measure space such that  $H$  is a reproducing kernel Hilbert space 
over $X$ and a closed subspace of $L^2(X,\mu)$ Then the multiplier algebra $\M(H)$ has property $\bA_1(1)$.
\end{thm}
In fact, Prunaru shows that $M(H)$ satisfies a much stronger predual factorization property known as $X(0,1)$.  In the particular case where $H$ is a Bergman space,  Bercovici proves a result similar to Theorem~\ref{T:prunaru} in \cite{Berc}.
 An application of Theorem~\ref{T:main} gives us what we need.

\begin{cor} \label{C:prun}
Let $(X, \B, \mu)$ be a $\sigma$-finite measure space such that  $H$  is both a reproducing kernel Hilbert space 
over $X$ and a closed subspace of $L^2(X,\mu)$. Suppose $\A$ is a dual algebra of multipliers on $H$.   If $x_1, \dots, x_n \in X$ and $w_1, \dots, w_n \in \bC$,
then there is a function $\phi \in \A$ with $\| \phi \|_\infty \leq 1$ and $\phi(x_i) = w_i$ for $i=1, \dots,  n$ if and only if 
\[ 
\left [
(1 - w_i\ol{w_j})k^\nu(x_i,x_j)
\right ]_{i,j=1}^n \geq 0
\]
for every measure of the form $\nu = |f|^2\mu$, for $f \in H$.
\end{cor}

Theorem~\ref{T:thm2} now follows from Corollary~\ref{C:prun} by taking $X$ to be a bounded domain $\Omega$ and setting $H = L^2_a(\Omega)$.
\bigskip

\small{\noindent \textbf{Acknowledgements}
The author wishes to express his gratitude towards  K.R. Davidson  and J. E. McCarthy for many helpful discussions and comments on this manuscript.

\end{document}